\documentclass[12pt,a4paper]{article}
\usepackage{amsmath}
\usepackage{amsfonts}
\usepackage{amsthm}
\usepackage{amscd}
\usepackage{amsbsy}
\usepackage{amsrefs}
\usepackage[latin2]{inputenc}
\usepackage{t1enc}
\usepackage[mathscr]{eucal}
\usepackage{indentfirst}
\usepackage{graphicx}
\usepackage{graphics}
\usepackage{pict2e}
\usepackage{epic}
\numberwithin{equation}{section}
\usepackage[margin=2.9cm]{geometry}
\usepackage{epstopdf} 
\usepackage{tikz-cd}
\usepackage{chngcntr}
\usepackage{amssymb}
\usepackage{esint}
\usepackage{hyperref}
\usepackage{mwe}
\usepackage{mathrsfs}
\usepackage{xfrac} 
\usepackage{verbatim}

\theoremstyle{plain}
\newtheorem{Th}{Theorem}[section]
\newtheorem{Lemma}[Th]{Lemma}

 \theoremstyle{definition}
\newtheorem{Def}[Th]{Definition}

\newtheorem{Rem}[Th]{Remark}
\newtheorem{?}[Th]{Problem}

\newcommand*\R{\mathbb{R}}

\newcommand*\Om{\Omega}

\newcommand*\delomega{\partial\Omega}

\newcommand{\Div}{\text{div}_x}
\newcommand{\vectoru}{\mathbf{u}}
\newcommand{\vectorU}{\mathbf{U}}
\newcommand{\vectorv}{\mathbf{v}}
\newcommand{\Epsilon}{\mathcal{E}}
\newcommand{\dx}{\text{ d}x}
\newcommand{\dt}{\text{ d}t\;}

\newcommand{\vectorphi}{\pmb{\varphi}}

\newcommand{\Nu}{\mathcal{V}}
\newcommand{\Nutx}{\mathcal{V}_{t,x}}
\newcommand{\DD}{\mathbb{D}}
\newcommand{\trace}{\text{tr}}
\begin{document}

\title{On weak (measure--valued)--strong uniqueness for compressible Navier-Stokes system with non-monotone pressure law}

\author{Nilasis Chaudhuri
	%\thanks{E-mail:\tt chaudhuri@math.tu-berlin.de}
}
\maketitle

\centerline{Technische Universit{\"a}t, Berlin}

\centerline{Institute f{\"u}r Mathematik, Stra\ss e des 17. Juni 136, D -- 10623 Berlin, Germany.}

\begin{abstract} 
	In this paper our goal is to define a \emph{renormalised dissipative measure--valued} (\text{rDMV}) \emph{solution of compressible Navier--Stokes system for fluids} with non-monotone pressure--density relation. We prove existence of rDMV solutions and establish a suitable relative energy inequality. Moreover we obtain the Weak (Measure-valued)-Strong uniqueness property of this rDMV solution with the help of relative energy inequality.
\end{abstract}



\section{Introduction}

Let $T>0$ and $\Omega \subset \mathbb{R}^d,\; d\in \{1,2,3\}$ be a bounded domain with smooth boundary. We consider the compressible Navier-Stokes equation in time-space cylinder $Q_T=(0,T)\times \Omega$ describing the time evolution of the mass density $\varrho=\varrho(t,x)$ and the velocity field $\vectoru=\vectoru(t,x)$ of a compressible viscous fluid:
\begin{itemize}
	\item \textbf{Conservation of Mass:}
	\begin{align}
	\partial_t \varrho + \text{div}_x (\varrho \mathbf{u})&=0. \label{cnse:cont}
	\end{align}
	\item \textbf{Conservation of Momentum:}
	\begin{align}
	\partial_t(\varrho \mathbf{u}) + \Div (\varrho \mathbf{u} \otimes \mathbf{u})+\nabla_x p(\varrho)&=\Div \mathbb{S}(\nabla_x \mathbf{u}). \label{cnse:mom}
	\end{align}
	\item  \textbf{Constitutive Relation:}
	Here $\mathbb{S}(\nabla_x \mathbf{u})$ is \textit{Newtonian stress tensor} defined by
	\begin{align}\label{cauchy_str}
	\mathbb{S}(\nabla_x \mathbf{u})=\mu \bigg(\frac{\nabla_x \vectoru + \nabla_x^{T} \vectoru}{2}-\frac{1}{d} (\Div\vectoru)\mathbb{I} \bigg) + \lambda (\Div \vectoru) \mathbb{I},
	\end{align}
	where $\mu >0$ and $\lambda >0$ are the \textit{shear} and \textit{bulk} viscosity coefficients, respectively.
	\item \textbf{Pressure Law:} In an isentropic setting, the pressure $p$ and the density $\varrho$ of the fluid are interrelated by :
		\begin{align}\label{p-condition}
		\begin{split}
		&p(\varrho)=h(\varrho)+q(\varrho), \text{ with } {q\in C_{c}^1(0,\infty)},\\
		&h\in C^1[0,\infty),\;  h(0)=0,\;h'>0,\;\text{in } (0,\infty),\;  \\
		&\lim_{\varrho\rightarrow \infty} \frac{h'(\varrho)}{\varrho^{\gamma-1}}= a>0 \text{ and } \gamma\geq 1.\\
		\end{split}
		\end{align} 
	\end{itemize}
\begin{Rem}
	The Consideration of $h$ in \eqref{p-condition} has been motivated from isentropic equation of state given by $h(\varrho)=a\varrho^\gamma$ with $\gamma\geq 1$ and $a>0$. 
\end{Rem}
 
 \begin{itemize}
 	\item Here we consider no slip boundary condition for velocity i.e.
 	\begin{align}\label{BC}
 	\vectoru|_{\{\partial\Om \times (0,T)\}=0}.
 	\end{align}
 \end{itemize}

 The compressible Navier--Stokes equations admit global--in--time weak solution(s) for general finite energy initial data and a large class of pressure--density constitutive relations. Considering $q\equiv 0$ in \eqref{p-condition} and following the literatures of Antontsev et al.\cite{AKM1983}, Lions\cite{PL1998}, Feireisl\cite{F2004b}, Plotnikov et al.\cite{PW2015} and many others, we observe global--in--time weak solution for adiabatic exponent $\gamma\geq 1$ for $d=1,2$ and $\gamma>\frac{3}{2}$ for $d=3$. {Even for non-monotone pressure, Feireisl in \cite{F2002} has proved a similar result and recent work by Bresch and Jabin \cite{BJ2018} indicates that for $p\in C^1[0,\infty)\geq0$, $p(0)=0$,  $\lim_{\varrho\rightarrow \infty} \frac{p'(\varrho)}{\varrho^{\gamma-1}}= a>0 \text{ and } \gamma\geq 2$ the system admits a weak solution.}
So it may seem unnecessary to develop the theory of measure valued solution that extends the class of generalised solutions but in the following discussion we will try to justify why we still choose to consider it.\\

The concept of measure valued solution to partial differential equation, more precisely for hyperbolic conservation law, was introduced by DiPerna in \cite{Di1985}. {The measure--valued solutions in the context of compressible Navier--Stokes solutions has been introduced by Neustupa in \cite{N1993}, drawing inspiration from Malek et al. \cite{MNRR1996}.}
 The topic has been revisited by Feireisl, Gwiazda, Swierczewska--Gwiazda and Wiedemann in \cite{FPAW2016}, where a suitable form of energy inequality have been introduced in the definition of the \emph{dissipative} measure--valued solutions.\\

 {Recently, the concept of measure--valued solutions has been studied again in the context of analysis of numerical schemes by Feireisl et al. in \cite{FL2018} and \cite{FLMS2018}. The crucial result in the aforementioned articles is the \emph{weak(measure--valued)--strong uniqueness principle} asserting that suitable (dissipative) measure valued and a strong solution starting from the same initial data necessarily coincide on the life--span of the latter. Identification of dissipative measure valued solution  as a limit of a given numerical scheme is easier and presence of weak(measure-valued)--strong uniqueness principle ensures the convergence of the scheme towards the strong solutions as long as the latter exists.}\\

 {In the work \cite{FPAW2016}, the corresponding Young measure describes oscillations of the density and velocity but handles the viscous term as a linear perturbation. In particular, the velocity gradient is not included in the Young measure. The weak--strong uniqueness principle can be established for a monotone pressure--density equation of state following 
the arguments used for the inviscid Euler system. Further their result proves existence of such a solution for any adiabatic exponent $\gamma\geq 1$ independent of dimension. }\\

{Weak--Strong uniqueness principle for monotone pressure has been proved by Feireisl et al. in \cite{FJN2012} and \cite{FNS2011} for weak solutions and in \cite{FPAW2016} for measure--valued solutions. Recently, weak--strong uniqueness principle in the class of weak solutions has been shown for the compressible Navier--Stokes system with a general non--monotone pressure density relation and/or the singular hard sphere pressure in \cite{F2018} and \cite{C2018}
. To prove the above mentioned results the key tool is the presence of viscosity. Hence this cannot be extended to an inviscid system like Euler system directly.}\\

 {To deal with the non-monotone pressure in Feireisl \cite{F2018} and \cite{C2018} the use of the renormalized version of the equation of continuity plays a crucial role. But that is \emph{non--linear} with respect to the velocity gradient {and density}. Extension of these results to the class of \emph{measure--valued} solutions therefore requires a new approach that incorporates the velocity gradient as an integral part of the associated Young measure in the spirit of B\v rezina et al. in \cite{BFN2018}.}\\
 
  It is the aim of the present paper to introduce a new concept of \emph{ renormalised dissipative measure valued}(rDMV) solutions for the compressible Navier--Stokes system that includes, in particular, the renormalized equation of continuity, and to show the weak--strong uniqueness principle in this class of a non--monotone pressure density state equation. {The plan for the paper is as follows:}
  	\begin{itemize}
  		\item {\bf Definition.} In section 2, we will introduce rDMV solutions, see \eqref{DMV defn}.
  		\item {\bf Existence.} In section 3, our goal is to show that an rDMV solution exists for any finite energy initial data, see Theorem \eqref{Exi-thm}.
  		
  		\item {\bf Weak-strong uniqueness.} In section 4, we prove that an rDMV solution coincides with the strong solution emanating from the same initial data on the life span of the latter, see Theorem \eqref{main-theorem}. 
  	\end{itemize}

\section{Definition of Measure valued solution}

Before going to our formal discussion define,
\textit{pressure potential } as :\\

\begin{itemize}
	\item When $p$ is given by \eqref{p-condition},
	\begin{align}\label{P-defn}
	\begin{split}
	&P(\varrho)= H(\varrho)+Q(\varrho) \text{ where}\\
	&H(\varrho)=\varrho \int_{1}^{\varrho} \frac{h(z)}{z^2} \; \text{d}z \text{ and }\; Q(\varrho)=\varrho \int_{1}^{\varrho} \frac{q(z)}{z^2} \; \text{d}z.
	\end{split}
	\end{align}	
	\item As a trivial consequence of above we obtain,
	\begin{align}\label{relation q and Q}
	\begin{split}
	&\varrho H^{'}(\varrho) -H(\varrho)= h(\varrho) \text{ and } \varrho H^{''} (\varrho)=h'(\varrho) \text{ for } \varrho>0,\\
	&\varrho Q^{'}(\varrho) -Q(\varrho)= q(\varrho) \text{ and } \varrho Q^{''} (\varrho)=q'(\varrho) \text{ for } \varrho>0.
	\end{split}
	\end{align}
	
\end{itemize}

\subsection{Phase Space:} We have discussed in the introduction that velocity gradient has been incorporated as a part of Young measure along with natural candidates for the phase space e.g. density and velocity $[\varrho,\vectoru]$. Hence a suitable phase space framework for the measure--valued solution is therefore
\begin{align}
\mathcal{F}=\{[s,\vectorv,\DD_{\vectorv}] \big | s\in [0,\infty),\; \vectorv\in \R^d,\; \DD_{\vectorv}\in \R^{d\times d}_{\text{sym}}\}.
\end{align}

\begin{Def}\label{DMV defn}
	We say that a parametrized measure $\{ \Nu_{t,x} \}_{(t,x)\in (0,T)\times \Omega}$,
	\begin{align*}
	\Nu \in L^{\infty}_{\text{weak}} \big( (0,T)\times \Om ;\mathcal{P}(\mathcal{F} )\big),
	\end{align*}
	is a renormalised dissipative measure--valued (rDMV) solution of Navier--Stokes system \eqref{cnse:cont}-\eqref{cauchy_str} in $(0,T)\times \Omega$, with the initial condition $\Nu_0$ and dissipation defect $\mathcal{D}$,
\begin{align*}
	\mathcal{D}\in L^{\infty}(0,T),\; \mathcal{D}\geq 0,
\end{align*}
	if the following holds.\\
	\begin{itemize}
		\item \textbf{Equation of Continuity:} For a.e. $\tau \in (0,T) $ and $\psi \in C^{1}([0,T]\times \bar{\Om})$ 
		 \begin{align} \label{mv continuity eqn}
		 \begin{split}
		 &\int_{ \Om} \langle \Nu_{\tau,x} ; s \rangle \psi(\tau, \cdot) \dx - \int_{ \Om} \langle \Nu_{0} ; s \rangle \psi(0, \cdot) \dx \\
		 &\quad = \int_{0}^{\tau} \int_{ \Om} \big[ \langle \Nu_{t,x}; s  \rangle \partial_{t}\psi + \langle \Nu_{t,x}; s \mathbf{v} \rangle \cdot \nabla_x \psi \big] \dx \dt.
		 \end{split}
		 \end{align}
	 	\item \textbf{Renormalized equation of continuty:} For a.e. $\tau \in (0,T) $ and $\psi \in C^{1}([0,T]\times \bar{\Om})$ we have
	 	\begin{align} \label{mv renorm continuity eqn}
	 	\begin{split}
	 	&\int_{ \Om} \langle \Nu_{\tau,x} ; b(s) \rangle \psi(\tau, \cdot) \dx - \int_{ \Om} \langle \Nu_{0} ; b(s) \rangle \psi(0, \cdot) \dx \\
	 	&\quad = \int_{0}^{\tau} \int_{ \Om} \big[ \langle \Nu_{t,x}; b(s)  \rangle \partial_{t}\psi + \langle \Nu_{t,x}; b(s) \mathbf{v} \rangle \cdot \nabla_x \psi \big] \dx \dt\\
	 	&\quad \quad - \int_{0}^{\tau} \int_{ \Om} \langle \Nu_{t,x} ;  (sb^\prime (s)-b(s)) \trace(\DD_{\vectorv}) \rangle \cdot \psi \dx \dt,
	 	\end{split}
	 	\end{align}
	 	where $b\in C^1[0,\infty),\; \exists r_b>0 \text{ such that }b'(x)=0,\; \forall x>r_b$. 
		\item \textbf{Momentum Equation:} 
		There exists a measure $r^{M} \in L^1([0,T];\mathcal{M}(\bar{\Om}))$ and $\xi\in L^{1}(0,T)$ such that for a.e. $\tau \in (0,T)$ and every $\vectorphi\in C^{1}([0,T]\times \bar{\Om};\R^N)$, $\vectorphi|_{\partial \Om} =0$,
		\begin{equation}
		\vert \langle r^M;\nabla_x \vectorphi \rangle \vert \leq \xi(\tau) \mathcal{D}(\tau) \Vert \vectorphi \Vert_{C^1(\bar{\Om})}
		\end{equation} 
		and
		\begin{align}\label{mv momentum eqn}
		\begin{split}
		&\int_{ \Om} \langle \Nu_{\tau,x}; s \mathbf{v} \rangle \cdot \vectorphi(\tau,\cdot) \dx - \int_{ \Om} \langle \Nu_{0} ; s\mathbf{v} \rangle \cdot \vectorphi(0,\cdot) \dx\\
		&= \int_{0}^{\tau} \int_{ \Om} \big[ \langle \Nu_{t,x}; s \mathbf{v} \rangle \cdot \partial_{t} \vectorphi +  \langle \Nu_{t,x}; s(\mathbf{v} \otimes \mathbf{v}) \rangle : \nabla_x \vectorphi + \langle \Nu_{t,x}; p(s) \rangle \Div \vectorphi  \big] \dx \dt\\
		& - \int_{0}^{\tau} \int_{ \Om} \langle \Nu_{t,x};\mathbb{S}(\DD_{\vectorv}) \rangle : \nabla_x \vectorphi \dx \dt + \int_{0}^{\tau} \langle r^M;\nabla_x \vectorphi \rangle \dt.
		\end{split}
		\end{align}
		\item \textbf{Momentum Compatibility:}
		This natural compatibility condition remains true.
		\begin{align}
		\begin{split}
		-\int_0^\tau \int_{ \Om} \langle \Nu_{t,x}; \vectorv \rangle \cdot \Div \mathbb{M} \dx \dt &= \int_{0}^\tau \int_{ \Om} \langle \Nu_{t,x}; \DD_{\vectorv} \rangle : \mathbb{M} \dx \dt \\
		&\text{ for any } \mathbb{M} \in C^1(\bar{Q}_T; \mathbb{R}^{d\times d}_{\text{sym}}).
		\end{split}
		\end{align}
		
		\item \textbf{Energy Inequality:}
		\begin{align}\label{mv energy inequality}
		\begin{split}
		\int_{ \Om} \big \langle\Nutx ; \big( \frac{1}{2}s\vert \mathbf{v} \vert^2 + P(s)\big) \big \rangle \dx &+ \int_{0}^\tau \int_{ \Om} \langle \Nu_{t,x};\mathbb{S}(\DD_{\vectorv}):\DD_{\vectorv} \rangle \dx \dt + \mathcal{D}(\tau)\\
		& \leq \int_{ \Om} \big \langle \Nu_{0} ; \big( \frac{1}{2}s\vert \mathbf{v} \vert^2 + P(s)\big) \big \rangle \dx,
		\end{split}
		\end{align}
		for a.e. $\tau \in (0,T)$.
		\item \textbf{Generalized Korn-- Poincar\'{e} inequality:}\\
		Let \begin{align}
		\mathbb{T}(A)=A+A^t-\frac{2}{d} \text{tr} (A) \mathbb{I}.
		\end{align}
		For $ \tilde{\vectoru}\in L^2(0,T;H^1_0(\Om;\R^d))$, the following inequality is true, 
		\begin{align}\label{poincare mv}
		\int_{0}^{\tau} \int_{ \Om} \langle \Nu_{t,x} ; \vert \mathbf{v} - \tilde{\vectoru} \vert^2 \rangle \dx \dt \leq c_P \int_{0}^{\tau} \int_{ \Om} \langle \Nu_{t,x} ; \vert \mathbb{T}(\DD_{\vectorv}) - \mathbb{T}(\nabla_x \tilde{\vectoru}) \vert^2 \rangle \dx \dt.
		\end{align}
	\end{itemize}
\end{Def}

 \begin{Rem}
 	Since here in this article our goal is to prove the weak-strong uniqueness, instead of considering initial condition as measure $\Nu_0$ we can directly consider \emph{finite energy initial data}. That means $(\langle \Nu_{0},s\rangle, \langle \Nu_{0};s\vectorv \rangle )=(\varrho_0, (\varrho\vectoru)_0)$ are functions with $\varrho_{0}\geq 0$, $(\varrho\vectoru)_0=0$ on the set $\{x\in \Om | \varrho_{0}(x)=0\}$ and 
 	\begin{align}\label{initial-data}
 	\int_{ \Om} \big( \frac{1}{2} \frac{ \vert (\varrho \vectoru)_{0} \vert^2}{\varrho_0} + P(\varrho_{0})  \big) (t,\cdot) \dx < \infty.
 	\end{align}
 
 \end{Rem}

\begin{Rem}
	As a consequence of the above definition, for $q \in C_c(0,\infty) $ and $ Q$ defined as in \eqref{relation q and Q} we have,
	\begin{align}
	\bigg[ \int_{ \Om} \langle \Nu_{t,x} ; Q(s) \rangle (t,\cdot) \dx\bigg]_{t=0}^{t=\tau} = -\int_{0}^{\tau} \int_{ \Om} \langle \Nu_{t,x} ; q(s) \trace(\DD_{\vectorv}) \rangle \dx \dt. 
	\end{align}
\end{Rem}

\begin{Rem}
	From the definition of $\mathbb{T}$ it follows,
	\begin{align}
	\mathbb{T}(A) : \mathbb{T}(A) = 2 \times \mathbb{T}(A):A,\; A\in \R^{d\times d}.
	\end{align} 
	
\end{Rem}

\section{Existence of solution}

From Feireisl \cite{F2002}, we have existence of weak solution for large adiabatic exponent $\gamma$ . Hence, this motivates the following approximate problem, 
\begin{align}
\partial_t \varrho + \text{div}_x (\varrho \mathbf{u})&=0, \label{cnse:cont:ap}\\
\partial_t(\varrho \mathbf{u}) + \Div (\varrho \mathbf{u} \otimes \mathbf{u})+\nabla_x p(\varrho)+\delta \nabla_x \varrho^\Gamma&=\Div \mathbb{S}(\nabla_x \mathbf{u}), \label{cnse:mom:ap}\\
u\vert_{\delomega}&=0. \label{bc:ap}
\end{align}
where $\delta>0$ is a small parameter, $\Gamma >1$ is large enough to ensure the existence of weak solution and $p$ follows \eqref{p-condition}. Further we assume that for the above mentioned problem, intial condition $\{ \varrho_{\delta,0}, (\varrho \vectoru)_{\delta,0} \}$ belongs to a certain regularity class for which weak solution exists. As an additional assuption we have,
\begin{align}
\begin{split}
\frac{1}{2} \varrho_{\delta,0} \vert \vectoru_{\delta,0} \vert^2 + P(\varrho_{\delta,0}) + \frac{\delta}{\Gamma-1} \varrho_{\delta,0}^\Gamma \rightarrow \frac{1}{2} \frac{\vert (\varrho \vectoru)_0 \vert^2 }{\varrho_0} + P(\varrho_0) \text{ in } L^1(\Om),
\end{split}
\end{align}
when $\delta \rightarrow 0$.

Thus we obtain  $$ \int_{ \Om} \big( \frac{1}{2} \varrho_{\delta,0} \vert \vectoru_{\delta,0} \vert^2 + P(\varrho_{\delta,0}) + \frac{\delta}{\Gamma-1} \varrho_{\delta,0}^\Gamma \big) (t,\cdot) \dx \leq c, $$ where $c$ is independent of $\delta$.

For each $\delta>0$, existence of weak solution $\{ \varrho_\delta, \vectoru_\delta\}_{\delta>0}$ for \eqref{cnse:cont:ap}-\eqref{cnse:mom:ap} with constitutive relation, initial and boundary condition follows directly from Feireisl \cite{F2002} for some $\Gamma\geq 2$.

Our goal is to verify that the family of weak solutions $\{ \varrho_\delta, \vectoru_\delta\}_{\delta>0}$ generates a dissipative measure-valued solution as defined in \eqref{DMV defn}.
\subsection{Apriori estimates:}
From the definition of dissipative weak solution we have the following estimates,

\begin{align}
\begin{split}
\sup_{  t\in [0,T]} \int_{ \Om} H(\varrho_\delta) (t,\cdot) \dx \leq c,\\
\sup_{  t\in [0,T]}\int_{ \Om} \varrho_\delta \vert \vectoru_\delta \vert^2 (t,\cdot)\dx \leq c,\\
\int_{0}^{T} \int_{ \Om} \mathbb{S}(\nabla_x \vectoru_\delta) : \nabla_x \vectoru_\delta \dx \dt \leq c,\\
\sup_{  t\in [0,T]} \frac{\delta}{\Gamma-1} \int_{ \Om} \varrho_\delta^{\Gamma}(t,\cdot)\dx \leq c.
\end{split}
\end{align}
\paragraph{}
By Korn inequality and Poincar\'{e} inequality we have $\vectoru_\delta$ is bounded in $L^2(0,T;W^{1,2}_{0}(\Om))$.
Further from \eqref{P-defn}, $\{\varrho_\delta \}$ is bounded in $L^{\infty}(0,T;L^{\gamma}(\Om))$ for $\gamma>1$ and $\{\varrho_\delta \log \varrho_\delta \}$ is bounded in $L^{\infty}(0,T; L^1(\Om))$ for $\gamma=1$. 

\paragraph{}
From our assumption $q\in C_c^1[0,\infty)$, we have $Q(\varrho)\approx \varrho$. Hence we can conclude that 
\begin{align}
\begin{split}
&\big[\frac{1}{2} \varrho_\delta \vert \vectoru_\delta \vert^2 + P(\varrho_\delta)\big] (t,\cdot) \in \mathcal{M}(\bar{\Om}) \text{ is bounded uniformly for }t\in (0,T),\\
&\big[ \mu \vert \nabla_x \vectoru_\delta \vert^2 + (\lambda-\frac{\mu}{d}) \vert \Div \vectoru_\delta \vert^2 \big] \text{ is bounded in }\mathcal{M}^{+}([0,T]\times \bar{\Om}),\\
&\delta \varrho_\delta^\Gamma (t,\cdot) \in \mathcal{M}^{+}(\bar{\Om}) \text{ is bounded uniformly for } t\in (0,T).
\end{split}
\end{align}

Thus passing to a subsequence, we obtain
\begin{align}
\begin{split}
&\big[\frac{1}{2} \varrho_\delta \vert \vectoru_\delta \vert^2 + P(\varrho_\delta)\big] (t,\cdot) \rightarrow E \text{ weakly-(*) in } L^{\infty}_{\text{weak}}(0,T;\mathcal{M}(\bar{\Om})),\\
&\big[ \mu \vert \nabla_x \vectoru_\delta \vert^2 + (\lambda-\frac{\mu}{d}) \vert \Div \vectoru_\delta \vert^2 \big] \rightarrow \sigma \text{ weakly-(*) in } \mathcal{M}^{+}([0,T]\times \bar{\Om}),\\
&\delta \varrho_\delta^\Gamma (t,\cdot) \rightarrow \zeta \text{ weakly-(*) in } L^{\infty}_{\text{weak}}(0,T;\mathcal{M}^{+}(\bar{\Om})).
\end{split}
\end{align}

Let $\Nu $ be a Young measure generated by $\{ \varrho_\delta, \vectoru_\delta, \mathbb{D}_{\vectoru_\delta}=\frac{\nabla_x\vectoru_\delta+ \nabla_x^{T} \vectoru_\delta}{2} \}_{\delta>0}$.\\
 Now we introduce two non-negative measures \\
 $E_{\infty}= E- \langle \Nu_{t,x}; \frac{1}{2}s \vert \vectorv \vert^2 + P(s) \rangle \dx $, $ \sigma_\infty = \sigma - \langle \Nu_{t,x}; \mathbb{S}(\DD_{\vectorv}):\DD_{\vectorv}  \rangle \dx \dt$.
 \subsection{Passage to limit}
 \subsubsection{Passage to limit in Energy inequality:}
 For the approximate problem \eqref{cnse:cont:ap}-\eqref{bc:ap} we have 
 \begin{align}
 \begin{split}
 \bigg[\int_{\Om} &\bigg( \frac{1}{2} \varrho_\delta \vert \vectoru_\delta \vert^2 +P(\varrho_\delta) + \frac{\delta}{\Gamma-1} \varrho_\delta^{\Gamma} \bigg)(t,\cdot) \dx \bigg]_{t=0}^{t=\tau}+ \int_{0}^{\tau} \int_{\Om} \mathbb{S}(\nabla_x \vectoru_\delta) :\nabla_x \vectoru_\delta \dx \dt \leq 0.\\
 \end{split}
 \end{align}
 Thus passing limit in the energy inequality, we obtain,
 \begin{align}
 \begin{split}
 &\int_{ \Om} \big \langle\Nu_{\tau,x} ; \big( \frac{1}{2}s\vert \mathbf{v} \vert^2 + P(s)\big) \big \rangle \dx + \int_{0}^\tau \int_{ \Om} \langle \Nu_{t,x};\mathbb{S}(\DD_{\vectorv}):\DD_{\vectorv} \rangle \dx \dt \\ &+ E_{\infty}(\tau)[\bar{\Om}]+ C\zeta(\tau)[\bar{\Om}] + \sigma_\infty[[0,\tau]\times \bar{\Om}]  \leq \int_{ \Om} \big \langle \Nu_{0} ; \big( \frac{1}{2}s\vert \mathbf{v} \vert^2 + P(s)\big) \big \rangle \dx,
 \end{split}
 \end{align}
  here $C>0$ fixed. Consider, 
  \begin{align}\label{defect_measure}
  \mathcal{D}(\tau)=E_{\infty}(\tau)[\bar{\Om}]+ C\zeta(\tau)[\bar{\Om}] + \sigma_\infty[[0,\tau]\times \bar{\Om}].
  \end{align}
  
  \subsubsection{Passage to limit in Renormalised Continuity equation:}
  We have,
  \begin{align}
  \begin{split}
  &\bigg[ \int_{ \Om} (\varrho_\delta+ b(\varrho_\delta)) \varphi \dx\bigg]_{t=0}^{t=\tau} \\
  &= 
  \int_0^{\tau} \int_{\Om} [ (\varrho_\delta+b(\varrho_\delta)) \partial_t \varphi + (\varrho_\delta+ b(\varrho_\delta)) \vectoru_\delta \cdot \nabla_x \varphi+ (b(\varrho_\delta)-\varrho_\delta b'(\varrho_\delta)\Div \vectoru_\delta \varphi] \dx \dt ,
  \end{split}
  \end{align} 
  where, $b\in C^1[0,\infty),\; \exists r_b>0 \text{ such that }b'(x)=0,\; \forall x>r_b$.
  This choice of $b$ implies that,
  \begin{align}
  \begin{split}
  (b(\varrho_\delta)-\varrho_\delta b'(\varrho_\delta)\Div \vectoru_\delta \in L^1((0,T)\times \Om)\text{ is uniformly bounded. }
  \end{split}
  \end{align} 
  
  Hence we obtain,
  \begin{align} 
  \begin{split}
  &\int_{ \Om} \langle \Nu_{\tau,x} ; s \rangle \psi(\tau, \cdot) \dx - \int_{ \Om} \langle \Nu_{0} ; s \rangle \psi(0, \cdot) \dx \\
  &\quad = \int_{0}^{\tau} \int_{ \Om} \big[ \langle \Nu_{t,x}; s  \rangle \partial_{t}\psi + \langle \Nu_{t,x}; s \mathbf{v} \rangle \cdot \nabla_x \psi \big] \dx \dt,
  \end{split}
  \end{align}
  and 
  	\begin{align}
  \begin{split}
  &\int_{ \Om} \langle \Nu_{\tau,x} ; b(s) \rangle \psi(\tau, \cdot) \dx - \int_{ \Om} \langle \Nu_{0} ; b(s) \rangle \psi(0, \cdot) \dx \\
  &\quad = \int_{0}^{\tau} \int_{ \Om} \big[ \langle \Nu_{t,x}; b(s)  \rangle \partial_{t}\psi + \langle \Nu_{t,x}; b(s) \mathbf{v} \rangle \cdot \nabla_x \psi \big] \dx \dt\\
  &\quad \quad - \int_{0}^{\tau} \int_{ \Om} \langle \Nu_{t,x} ;  (sb^\prime (s)-b(s)) \trace(\DD_{\vectorv}) \rangle \cdot \psi \dx \dt.
  \end{split}
  \end{align}
  
  \subsubsection{Passage to limit in Momentum equation:}
  We have,
  \begin{align}
  \begin{split}
  &\bigg[\int_{\Om} \varrho_\delta \vectoru_\delta(\tau,\cdot)\cdot \vectorphi(\tau,\cdot) \dx\bigg]_{t=0}^{t=\tau} \\
  &=\int_0^{\tau}\int_{\Om} [\varrho_\delta \vectoru_\delta \cdot \partial_{t} \vectorphi + (\varrho_\delta \vectoru_\delta \otimes \vectoru_\delta : \nabla_x \vectorphi + (p(\varrho_\delta) + \delta \varrho_\delta^\Gamma) \Div \vectorphi -\mathbb{S}(\nabla_x \vectoru_\delta):\nabla_x \vectorphi] \dx \dt .\\		
  \end{split}
  \end{align}
   Using $\varrho_\delta u_{\delta,i} u_{\delta,j} \leq \varrho_\delta \vert \vectoru_\delta \vert^2$, $p(\varrho_\delta) \lessapprox P(\varrho_\delta)$ and Lemma 2.1 from Feireisl et al. \cite{FPAW2016} we obtain
   \begin{align}
   \begin{split}
   &\int_{ \Om} \langle \Nu_{\tau,x}; s \mathbf{v} \rangle \cdot \vectorphi(\tau,\cdot) \dx - \int_{ \Om} \langle \Nu_{0} ; s\mathbf{v} \rangle \cdot \vectorphi(0,\cdot) \dx\\
   &= \int_{0}^{\tau} \int_{ \Om} \big[ \langle \Nu_{t,x}; s \mathbf{v} \rangle \cdot \partial_{t} \vectorphi +  \langle \Nu_{t,x}; s(\mathbf{v} \otimes \mathbf{v}) \rangle : \nabla_x \vectorphi + \langle \Nu_{t,x}; p(s) \rangle \Div \vectorphi  \big] \dx \dt\\
   & - \int_{0}^{\tau} \int_{ \Om} \langle \Nu_{t,x};\mathbb{S}(\DD_{\vectorv}) \rangle : \nabla_x \vectorphi \dx \dt + \int_{0}^{\tau} \langle r^M;\nabla_x \vectorphi \rangle \dt+ \int_{0}^{\tau} \langle r^L;\Div \vectorphi \rangle \dt.
   \end{split}
   \end{align}
    Here, $r^M=\{r^M_{i,j}\}_{i,j=1}^{d}, r^M_{i,j}\in L^{\infty}_{{\text{weak}}}(0,T;\mathcal{M}(\bar{\Om}))$ and $r^L\in L^{\infty}_{{\text{weak}}}(0,T;\mathcal{M}(\bar{\Om}))$ such that $$ \vert r^M_{i,j}(\tau) \vert \leq E_{\infty}(\tau) \text{ and } \vert r^L(\tau) \vert \leq \zeta (\tau).$$
    
    $r^M$ and $r^L$ contain the concentration defect of the terms $\varrho_\delta \vectoru_\delta \otimes \vectoru_\delta$, $p(\varrho_\delta)$ and $ \delta \varrho_\delta^\Gamma $. By virtue of \eqref{defect_measure},  these are controlled by $\mathcal{D}$.
   \subsubsection{Verification of Momentum compatibility:}
   Since $\vectoru_\delta$ is bounded in $L^2(0,T;W^{1,2}_{0}(\Om))$, in this case we can check the relation easily.
   
   \subsubsection{Verification of Generalized Korn-- Poincar\'{e} inequality:}
   It can be proved along similar line as in Br\v ezina et al. \cite{BFN2018}.\\
  \subsection{Main Theorem:}
  We conclude this section with the following theorem,  
   \begin{Th}\label{Exi-thm}
   	Suppose $\Om$ is a regular bounded domain in $\mathbb{R}^d$ with $d=1,2,3$ and suppose the pressure
   	satisfies \eqref{p-condition}. If $(\varrho_0, (\varrho \vectoru)_0)$ satisfies \eqref{initial-data}, then there exists a dissipative measure-valued solution as defined in \eqref{DMV defn} with initial data $\Nu_{0} = \delta_{\{ \varrho_0, (\varrho \vectoru)_0 \}}$.
   \end{Th}

\section{Relative Energy and Weak-Strong Uniqueness}
Relative energy was first introduced by Dafermos in \cite{D1979} in the context of hyperbolic conservation laws. In the context of compressible Navier--Stokes it had been introduced by Feireisl, Jin, Novotn\'y and Sun in \cite{FJN2012} and \cite{FNS2011}. Motivated from the relative energy mentioned in those articles for weak solutions to barotropic Navier-Stokes system, i.e.
\begin{align}\label{rel ent weak}
\Epsilon(t)=\Epsilon(\varrho,\vectoru \vert r,\vectorU)(t):= \int_{\Omega}\frac{1	}{2} \varrho \vert \vectoru-\vectorU\vert^2 + (H(\varrho)-H(r) -H^{'}(r)(\varrho -r)) (t,\cdot) \dx ,
\end{align} 
we define,
\begin{align}\label{rel ent mv}
\mathcal{E}_{mv}(\varrho,\vectoru \vert r,\vectorU)(t):= \int_{ \Om} \bigg[ \big \langle \Nu_{t,x}; \frac{1}{2}s\vert \mathbf{v}- \vectorU \vert^2 + H(s)-H(r)-H'(r)(s-r)\big \rangle\bigg] \dx,
\end{align}
where $r,\vectorU$ are arbitrary test functions and $\{\varrho,\vectoru\}$ in \eqref{rel ent weak} is weak solution of \eqref{cnse:cont}-\eqref{cauchy_str}, while in \eqref{rel ent mv} $\Nu$ is a solution as defined in \eqref{DMV defn}.
 
\begin{Lemma}
	Let $( \Nu, \mathcal{D})$ be a measure valued solution of \eqref{cnse:cont}-\eqref{cauchy_str} for initial data $\Nu_{0}$ by definition \eqref{DMV defn}. Then for smooth $r,\vectorU$, we have the following relative energy inequality:
	\begin{align}\label{re1}
	\begin{split}
	\mathcal{E}_{mv}(\tau) + & \int_{0}^\tau \int_{ \Om} \langle \Nu_{t,x};\mathbb{S}(\DD_{\vectorv})  : \DD_{\vectorv} \rangle \dx \dt - \int_{0}^\tau \int_{ \Om} \langle \Nu_{t,x};\mathbb{S}(\DD_{\vectorv}) \rangle : \nabla_x \vectorU \dx \dt + \mathcal{D}(\tau) \\
	&\leq \int_{ \Om} \bigg[ \big \langle \Nu_{0,x}; \frac{1}{2}s\vert \mathbf{v}- \vectorU_0 \vert^2 + H(s)-H(r_0)-H'(r_0)(s-r_0)\big \rangle\bigg] \dx\\
	& - \int_0^{\tau} \int_{ \Om} \langle \Nu_{\tau,x}; s \vectorv \rangle \cdot \partial_{t}\vectorU \dx \dt\\
	&- \int_0^{\tau} \int_{ \Om} [\langle \Nu_{t,x}; s\vectorv \otimes \vectorv \rangle : \nabla_x \vectorU + \langle \Nu_{t,x} ; h(s)\rangle\Div\vectorU ] \dx \dt\\
	&+\int_0^{\tau} \int_{ \Om} [\langle \Nu_{t,x};s\rangle \vectorU \cdot \partial_{t} \vectorU + \langle \Nu_{t,x} ; s\vectorv \rangle \cdot (\vectorU \cdot \nabla_x)\vectorU] \dx \dt\\
	&+\int_{0}^{\tau} \int_{ \Om}\bigg[ \big\langle \Nu_{t,x};(1-\frac{s}{r}) \big \rangle h'(r) \partial_t r - \langle \Nu_{t,x}; s\vectorv \rangle \cdot \frac{h'(r)}{r} \nabla_x r \bigg] \dx \dt\\
	&-\int_{0}^{\tau} \int_{ \Om} \langle \Nu_{t,x} ; q(s)\rangle \Div \vectorU \dx \dt+\int_{0}^{\tau} \int_{ \Om} \langle \Nu_{t,x} ; q(s) \trace (\DD_{\vectorv})\rangle \dx \dt\\
	& -\int^{\tau}_{0} \langle r^M ; \nabla_x \vectorU \rangle \dt,
	\end{split}
	\end{align}
	here, $\vectorU_0(x)=\vectorU(0,x)$ and $r_0(x)=r(0,x)$ for $x\in \Om$.
\end{Lemma}
\begin{proof}
	By direct calculation we can show that,
	\begin{align}
	\begin{split}
	\Epsilon_{mv}(\tau) &=\int_{ \Om} \bigg\langle \Nu_{\tau,x} ; \frac{1}{2}s\vert \vectorv \vert ^2 + H(s)\bigg\rangle \dx - \int_{ \Om} \langle \Nu_{\tau,x}; s \vectorv \rangle \cdot \vectorU \dx \\
	& \quad +\int_{ \Om} \frac{1}{2} \langle \Nu_{\tau,x}; s \rangle \vert \vectorU \vert^2 \dx - \int_{ \Om}\langle\Nu_{\tau,x};s \rangle H'(r) \dx + \int_{ \Om} h(r) \dx = \Sigma_{i=1}^{5} K_i.
	\end{split}
	\end{align}
	
	Now we look for the terms $K_i$ for $i=1(1)5$. We have some bound for $K_1$ from \eqref{mv energy inequality}. To estimate $K_2$ we use \eqref{mv momentum eqn} and for $K_3,K_4$ we use \eqref{mv continuity eqn}. Calculating these terms we will obtain the desired result.
\end{proof}

\subsection{Main Theorem:}
We state the main theorem
\begin{Th}\label{main-theorem}
	Let $\Om \subset \R^d,\; d=1,2,3$ be a smooth bounded domain. Suppose the pressure $p$ satisfies \eqref{p-condition}. Let $\{ \Nu_{t,x}, \mathcal{D} \}$ be a dissipative measure-valued solution to the barotropic Navier-Stokes system \eqref{cnse:cont}-\eqref{cauchy_str} in $(0,T)\times \Omega$, with initial state represented by $\Nu_{0}$, in the sense specified in Definition \eqref{DMV defn}. Let $\{r,U\}$ be a strong solution to \eqref{cnse:cont}-\eqref{cauchy_str} in $(0,T)\times \Omega$ belonging to the class
	\begin{align}\label{reg_class_str_sol}
	r,\;\nabla_x r,\; \vectorU,\; \nabla_x \vectorU \in C([0,T]\times \bar{\Om}),\; \partial_{t} \vectorU \in L^2(0,T; C(\bar{\Om};\R^d)),\; r>0,\; \vectorU|_{\partial\Om}=0.
	\end{align}
	Then there is a constant $\Lambda=\Lambda(T)$, depending only the norms of $r,\;r^{-1},\; \vectorU,\;$ and $\xi$ in the aforementioned spaces, such that
	
	\begin{align}
	\begin{split}
	\int_{ \Om}& \bigg[ \big \langle \Nu_{\tau,x}; \frac{1}{2}s\vert \mathbf{v}- \vectorU \vert^2 + H(s)-H(r)-H'(r)(s-r)\big \rangle\bigg] \dx + \mathcal{D}(\tau)\\
	&\;\leq \Lambda(T) \int_{ \Om}  \bigg[ \big \langle \Nu_{0,x}; \frac{1}{2}s\vert \mathbf{v}- \vectorU(0,\cdot) \vert^2 + H(s)-H(r(0,\cdot))-H'(r_0)(s-r(0,\cdot))\big \rangle\bigg] \dx,
	\end{split}
	\end{align}
	
	for a.e. $\tau \in (0,T)$. In particular, if the initial states coincide, i.e.
	\begin{align*}
	\Nu_{0,x}=\delta_{\{r(0,x),U(0,x)\}}, \;\text{for a.e. }x\in \Om
	\end{align*}
	then $\mathcal{D}=0$, and
	\begin{align*}
	\Nu_{\tau,x}= \delta_{\{r(\tau,x),U(\tau,x)\}} \;\text{for a.e. }\tau 
	\in(0,T)\;\text{for a.e. }x\in \Om.
	\end{align*}
\end{Th}

%\lessapprox%	

From now on our goal is to prove the aforementioned theorem. We assume $\{r,\vectorU\}$ solves \eqref{cnse:cont}-\eqref{cauchy_str} and belongs to regularity class \eqref{reg_class_str_sol}. Further to simplify calculation, we assume 
\begin{align}\label{simple reg class}
\partial_{t} \vectorU \in C(\bar{Q}_T;R^d).
\end{align}

 Then we rewrite \eqref{re1} as,

\begin{align}\label{re3}
\begin{split}
\mathcal{E}_{mv}(\tau) + & \int_{0}^\tau \int_{ \Om} \langle \Nu_{t,x};\mathbb{S}(\DD_{\vectorv}) : \DD_{\vectorv} \rangle \dx \dt + \mathcal{D}(\tau) \\
&- \int_{0}^\tau \int_{ \Om} \langle \Nu_{t,x};\mathbb{S}(\DD_{\vectorv}) \rangle : \nabla_x \vectorU \dx \dt \\
&-\int_{0}^\tau \int_{ \Om} \mathbb{S}(\nabla_x \vectorU) : \langle \Nu_{t,x};( \DD_{\vectorv} - \nabla_x \vectorU )\rangle \dx \dt \\
& \leq \int_{ \Om} \bigg[ \big \langle \Nu_{0,x}; \frac{1}{2}s\vert \mathbf{v}- \vectorU_0 \vert^2 + H(s)-H(r_0)-H'(r_0)(s-r_0)\big \rangle\bigg] \dx\\
& + \int_0^{\tau} \int_{ \Om} \langle \Nu_{t,x}; (s(\vectorv-\vectorU)\cdot \nabla_x )\vectorU \cdot (\vectorU-\vectorv) \rangle \dx \dt\\
&  + \int_0^{\tau} \int_{ \Om} \langle \Nu_{t,x} ; (s-r)(\vectorU - \vectorv) \rangle \cdot \frac{1}{r}\big( \Div \mathbb{S}(\nabla_x \vectorU)- \nabla_x q(r) \big) \dx \dt  \\
&  +  \int_0^{\tau} \int_{ \Om} \langle \Nu_{t,x} ; (- h(s) +h(r)+ h'(r)(s -r)) \rangle \; \Div \vectorU\;   \dx \dt\\
&  + \int_0^{\tau} \int_{ \Om} \big\langle \Nu_{t,x} ; \big(q(s)-q(r)\big )(\trace(\DD_{\vectorv})- \Div \vectorU)  \big\rangle \dx \dt\\
&  + \Vert \vectorU \Vert_{C^1([0,T]\times \bar{\Om}; \R^{N})} \int^{\tau}_{0}\xi(t) \mathcal{D}(t) \dt.
\end{split} 
\end{align}
We have,

\begin{align}\label{re4}
\begin{split}
\mathcal{E}_{mv}(\tau) + & \int_{0}^\tau \int_{ \Om} \langle \Nu_{t,x};\mathbb{S}(\DD_{\vectorv}-\nabla_x \vectorU) : \DD_{\vectorv} - \nabla_x \vectorU  \rangle \dx \dt + \mathcal{D}(\tau) \\
&\leq  \int_{ \Om} \bigg[ \big \langle \Nu_{0,x}; \frac{1}{2}s\vert \mathbf{v}- \vectorU_0 \vert^2 + H(s)-H(r_0)-H'(r_0)(s-r_0)\big \rangle\bigg] \dx\\
& + \int_0^{\tau} \int_{ \Om} \langle \Nu_{t,x}; (s(\vectorv-\vectorU)\cdot \nabla_x )\vectorU \cdot (\vectorU-\vectorv) \rangle \dx \dt\\
&  + \int_0^{\tau} \int_{ \Om} \langle \Nu_{t,x} ; (s-r)(\vectorU - \vectorv) \rangle \cdot \frac{1}{r}\big( \Div \mathbb{S}(\nabla_x \vectorU)- \nabla_x q(r) \big) \dx \dt  \\
&  +  \int_0^{\tau} \int_{ \Om} \langle \Nu_{t,x} ; (- h(s) +h(r)+ h'(r)(s -r)) \rangle \; \Div \vectorU\;   \dx \dt\\
&  + \int_0^{\tau} \int_{ \Om} \big\langle \Nu_{t,x} ; \big(q(s)-q(r)\big )(\trace(\DD_{\vectorv})- \Div \vectorU)  \big\rangle \dx \dt\\
&  + \Vert \vectorU \Vert_{C^1([0,T]\times \bar{\Om}; \R^{N})} \int^{\tau}_{0}\xi(t) \mathcal{D}(t) \dt.
\end{split} 
\end{align}

 Using the relation between $\mathbb{S}$ and $\mathbb{T}$ we obtain,

 \begin{align}
 \begin{split}
 \mathcal{E}_{mv}(\tau) &+  \frac{\mu}{4}\int_{0}^\tau \int_{ \Om} \langle \Nu_{t,x};\mathbb{T}(\DD_{\vectorv}-\nabla_x \vectorU) : \mathbb{T}(\DD_{\vectorv}-\nabla_x \vectorU)  \rangle \dx \dt \\
 & + \lambda \int_{0}^\tau \int_{ \Om} \langle \Nu_{t,x};\vert \trace(\DD_{\vectorv})-\Div \vectorU\vert^2 \rangle \dx \dt + \mathcal{D}(\tau) \\
  &\leq \int_{ \Om} \bigg[ \big \langle \Nu_{0,x}; \frac{1}{2}s\vert \mathbf{v}- \vectorU_0 \vert^2 + H(s)-H(r_0)-H'(r_0)(s-r_0)\big \rangle\bigg] \dx\\
 & + \int_0^{\tau} \int_{ \Om} \langle \Nu_{t,x}; (s(\vectorv-\vectorU)\cdot \nabla_x )\vectorU \cdot (\vectorU-\vectorv) \rangle \dx \dt\\
 &  + \int_0^{\tau} \int_{ \Om} \langle \Nu_{t,x} ; (s-r)(\vectorU - \vectorv) \rangle \cdot \frac{1}{r}\big( \Div \mathbb{S}(\nabla_x \vectorU)- \nabla_x q(r) \big) \dx \dt  \\
 &  +  \int_0^{\tau} \int_{ \Om} \langle \Nu_{t,x} ; (- h(s) +h(r)+ h'(r)(s -r)) \rangle \; \Div \vectorU\;   \dx \dt\\
 &  + \int_0^{\tau} \int_{ \Om} \big\langle \Nu_{t,x} ; \big(q(s)-q(r)\big )(\trace(\DD_{\vectorv})- \Div \vectorU)  \big\rangle \dx \dt\\
 &  + \Vert \vectorU \Vert_{C^1([0,T]\times \bar{\Om}; \R^{N})} \int^{\tau}_{0}\xi(t) \mathcal{D}(t) \dt = \Sigma_{i=1}^{6} \mathcal{I}_i.
 \end{split} 
 \end{align}

We know that from pressure and density relation \eqref{p-condition} we have,

\begin{Lemma}\label{p by P 1}
	Suppose $H$  is defined as in \eqref{P-defn} and $r$ lies on a compact subset of $(0,\infty)$. Then we have,
	\begin{align}
	H(\varrho)-H(r)-H'(r)(\varrho -r) \geq c(r)  
	\begin{cases}
	&(\varrho -r)^2 \text{ for } r_1 \leq \varrho \leq r_2,\\
	&(1+ \varrho^{\gamma})  \text{ otherwise }\\
	\end{cases},
	\end{align}
	where $c(r)$ is uniformly bounded for $r$ belonging to compact subsets of $(0,\infty)$. 
\end{Lemma}
 \begin{Rem}\label{p byP rem}
 	We choose $r_1 < \frac{\inf r}{2}$, $ r_2 > 2 \; {\sup r}$ and $1 + \varrho^\gamma \geq \max\{ \varrho, \varrho^2 \},\; \forall \varrho \geq r_2$.  
 \end{Rem}
Hence for \eqref{p-condition} we have,
 \begin{Lemma}
	For $\varrho\geq 0$,
	\begin{align}
	\vert h(\varrho)-h(r) -h'(r) (\varrho -r) \vert \leq C(r)(  H(\varrho)-H(r)-H'(r)(\varrho -r) ),
	\end{align}
	where $C(r)$ is uniformly bounded if $r$ lies in some compact subset of $(0,\infty)$.
\end{Lemma}
Next we have to estimate $\mathcal{I}_i$ for $i=2(1)5$ of \eqref{re3}. 
\begin{itemize}
	\item \textbf{Remainder term $\mathcal{I}_2$:}\\
	We have,
	\begin{align}
	\vert \mathcal{I}_2 \vert \leq \Vert \vectorU \Vert_{C^1([0,T]\times \bar{\Om}; \R^{N})} \int_{0}^{\tau} \Epsilon_{mv}(t) \dt.
	\end{align}
	\item \textbf{Remainder term $\mathcal{I}_4$:}\\
	Similarly using lemma we obtain,
	\begin{align}
	\vert \mathcal{I}_4 \vert \leq C \int_{0}^{\tau} \Epsilon_{mv}(t) \dt.
	\end{align}
	\item \textbf{Remainder term $\mathcal{I}_3$:}\\
	Let $\text{supp}(q)=[q_1,q_2]$.
	Consider $r_1<\min\{\frac{q_1}{2}, \frac{1	}{2}\inf r\}$ and $r_2>\max \{2 q_2, 2 \times \sup r \}$.  
	We consider $\psi \in C_c^{\infty}(0,\infty),\; 0\leq \psi \leq 1,\; \psi(s)=1 \text{ for }s\in (r_1,r_2)$.	
	Now, 
	\begin{align*}
	&\langle \Nu_{t,x} ; (s-r)(\vectorU-\vectorv) \rangle\\
	&=\langle \Nu_{t,x}; \psi(s)(s-r)(\vectorU-\vectorv)\rangle+\langle \Nu_{t,x}; (1-\psi(s))(s-r)(\vectorU-\vectorv)\rangle
	\end{align*}
	
	Consequently we obtain 
	\begin{align}
	\begin{split}
	\langle \Nu_{t,x}; \psi(s) (s-r)(\vectorU-\vectorv) \rangle \leq \frac{1}{2} \bigg \langle \Nu_{t,x}; \frac{\psi^2(s)}{\sqrt{s}} (s-r)^2 \bigg \rangle + \frac{1}{2} \bigg \langle \Nu_{t,x}; \frac{\psi^2(s)}{\sqrt{s}} s\vert \vectorU-\vectorv \vert^2 \bigg \rangle.
	\end{split}
	\end{align}

	Now using that $\psi$ is compactly supported in $(0,\infty)$ and  lemma \eqref{p by P 1} we conclude that,
	\begin{align}
	\begin{split}
	&\int_0^{\tau} \int_{ \Om} \langle \Nu_{t,x} ; \psi(s) (s-r)(\vectorU - \vectorv) \rangle \cdot \frac{1}{r}\big( \Div \mathbb{S}(\nabla_x \vectorU)- \nabla_x q(r) \big) \dx \dt \\
	& \leq \Vert \frac{1}{r}\big( \Div \mathbb{S}(\nabla_x \vectorU)- \nabla_x q(r) \big) \Vert_{C([0,T]\times \bar{\Om}; \R^{N})} \int_{0}^{\tau} \Epsilon_{mv}(t) \dt.
	\end{split}
	\end{align}
	
	We rewrite $1-\psi(s)=w_1(s)+w_2(s)$, where $\text{supp}(w_1)\subset [0,r_1)$ and $\text{supp}(w_2)\subset (0,r_2]$,
	
	\begin{align*}
	&\langle \Nu_{t,x}; (1-\psi(s)) (s-r)(\vectorU-\vectorv) \rangle=\langle \Nu_{t,x}; (w_1(s)+w_2(s)) (s-r)(\vectorU-\vectorv) \rangle.
	\end{align*}
	
	For $\delta>0$ we obtain,
	\begin{align*}
	\langle \Nu_{t,x}; w_1(s) (s-r)(\vectorU-\vectorv) \rangle\leq c(\delta) \langle \Nu_{t,x}; w_1^2(s)(s-r)^2  \rangle + \delta \langle \Nu_{t,x}; \vert \vectorU - \vectorv \vert^2 \rangle.
	\end{align*}
	The first term on the right hand side is controlled by $\Epsilon$ meanwhile the second term can be absorbed in the left hand side of \eqref{re3} by virtue of generalised Korn-Poincar{\'e}  inequality as in \eqref{poincare mv}. Then we have,
	\begin{align*}
	\langle \Nu_{t,x}; w_1(s) (s-r)(\vectorU-\vectorv) \rangle\leq& C \int_{0}^{\tau} \Epsilon_{mv}(t) \dt \\
	&+ \delta c_P \int_{0}^{\tau} \int_{ \Om} \langle \Nu_{t,x} ; \vert \mathbb{T}(\DD_{\vectorv}) - \mathbb{T}(\nabla_x { \vectorU}) \vert^2 \rangle \dx \dt.
	\end{align*}
	
	Now, 
	\begin{align*}
	\langle \Nu_{t,x}; w_2(s) (s-r)(\vectorU-\vectorv) \rangle\leq c \langle \Nu_{t,x}; w_2(s)(s+ s  \vert \vectorU - \vectorv \vert^2)\rangle.
	\end{align*}
	In this inequality both integrals can be controlled by $\Epsilon_{mv}$.\\
	We take $\delta$ small enough and combine all the above terms to obtain,
	\begin{align}
	\begin{split}
	\vert \mathcal{I}_3 \vert \leq c \int_{0}^{\tau} \Epsilon_{mv}(t) \;\dt + \frac{\mu}{8}\int_{0}^{\tau} \int_{ \Om} \langle \Nu_{t,x} ; \vert \mathbb{T}(\DD_{\vectorv}) - \mathbb{T}(\nabla_x { \vectorU}) \vert^2 \rangle \dx \dt.
	\end{split}
	\end{align}

	\item \textbf{Remainder term $\mathcal{I}_5$:}\\
	For $\epsilon>0$ we have,
	\begin{align*}
	&\big\langle \Nu_{t,x} ; \big(q(s)-q(r)\big )(\trace(\DD_{\vectorv})- \Div \vectorU)  \big\rangle\\
	 &\leq \frac{1}{4\epsilon}\big\langle \Nu_{t,x} ; \big(q(s)-q(r)\big )^2  \big\rangle+ \epsilon \big\langle \Nu_{t,x} ; \vert \trace(\DD_{\vectorv})- \Div \vectorU\vert^2  \big\rangle,\; \text{ for } \epsilon>0.
	\end{align*}
	Further using a similar $\psi$ we obtain,
	
	\begin{align}
	\begin{split}
	\vert \mathcal{I}_5 \vert \leq c \int_{0}^{\tau} \Epsilon_{mv}(t) \;\dt + \frac{\lambda}{2}\int_{0}^{\tau} \int_{ \Om} \langle \Nu_{t,x} ; \vert \trace(\DD_{\vectorv}) - \Div\vectorU \vert^2 \rangle \dx \dt.
	\end{split}
	\end{align}
	
\end{itemize}

\subsubsection{Proof of the theorem \eqref{main-theorem}:}
Considering the above discussion, additional assumption \eqref{simple reg class} and combining all $\mathcal{I}_i$ for $i=1(1)6$, we obtain,
\begin{align}\label{re6}
\begin{split}
\mathcal{E}_{mv}(\tau) &+  \frac{\mu}{8}\int_{0}^\tau \int_{ \Om} \langle \Nu_{t,x};\mathbb{T}(\DD_{\vectorv}-\nabla_x \vectorU) : \mathbb{T}(\DD_{\vectorv}-\nabla_x \vectorU)  \rangle \dx \dt \\
& + \frac{\lambda}{2} \int_{0}^\tau \int_{ \Om} \langle \Nu_{t,x};\vert \trace(\DD_{\vectorv})-\Div \vectorU\vert^2 \rangle \dx \dt + \mathcal{D}(\tau) \\
&\leq \int_{ \Om} \bigg[ \big \langle \Nu_{0,x}; \frac{1}{2}s\vert \mathbf{v}- \vectorU_0 \vert^2 + H(s)-H(r_0)-H'(r_0)(s-r_0)\big \rangle\bigg] \dx\\
&\; +C(r,\vectorU,q)  \int_0^\tau \mathcal{E}_{mv}(t)\; \dt + \int_0^\tau \xi(t) \mathcal{D}(t) \; \dt  
\end{split} 
\end{align}

Now applying Gr{\"o}nwall's lemma, we conclude, 
\begin{align}
\begin{split}
&\mathcal{E}_{mv}(\varrho,\vectoru \vert r,\vectorU)(t) + \mathcal{D}(t) \\
&\leq c(T) \int_{ \Om}  \bigg[ \big \langle \Nu_{0,x}; \frac{1}{2}s\vert \mathbf{v}- \vectorU(0,\cdot) \vert^2 + H(s)-H(r(0,\cdot))-H'(r_0)(s-r(0,\cdot))\big \rangle\bigg] \dx
\end{split}
\end{align}
for a.e. $\tau \in [0,T]$.\\
\begin{Rem}
For simplicity of the proof we assume \eqref{simple reg class}. If we stick to only \eqref{reg_class_str_sol} then we have $ \int_0^\tau \eta_{(r,\vectorU,q)}(t)  \mathcal{E}_{mv}(t)\; \dt$ where $\eta_{(r,\vectorU,q)} \in L^1(0,T)$ instead of the term
$C(r,\vectorU,q)  \int_0^\tau \mathcal{E}_{mv}(t)\; \dt $ in \eqref{re6}.  	
\end{Rem}

\section{Concluding remarks}

 Instead of considering $q\in C_c^{1}(0,\infty)$ if we assume $q\in C^1$ with $q'(\varrho) \approx \varrho^\alpha$ as $\varrho \rightarrow \infty$, then we can obtain Weak (measure-valued)-Strong uniqueness principle if $\alpha+1\leq \frac{\gamma}{2}$. Even if $q$ is a globally Lipschitz function in $[0,\infty)$ we have the same principle for $\gamma\geq 2$. Further existence of such an rDMV solution can be generated by limit of weak solutions of the approximate problem \eqref{cnse:cont:ap}-\eqref{cnse:mom:ap} whose existence can be guaranteed by the work of Bresch and Jabin in \cite{BJ2018}.

\vspace{5mm}

\centerline{ \bf Acknowledgement} 
\vspace{2mm}

The work was supported by Einstein Stiftung, Berlin. I would like to thank my  Ph.D supervisor Prof. E. Feireisl for his valuable suggestions and comments.

%	\bibliography{nilasis_ws}
%	\bibliographystyle{unsrt}

\end{document}